\documentclass[%
aip,
amsmath,amssymb,
reprint,%
author-numerical,%
Conference Proceedings
]{revtex4-1}

\usepackage{graphicx,color,amsthm,xspace}% Include figure files
\usepackage{dcolumn}% Align table columns on decimal point
\usepackage{bm}% bold math
\usepackage[utf8]{inputenc}
\usepackage[T1]{fontenc}
\usepackage{mathptmx}

\newtheorem{lemma}{Lemma}

\newtheorem{cor}{Corollary}
\newtheorem{theorem}{Theorem}
\newtheorem{prop}{Proposition}

\newtheorem{assumption}{Assumption}

\newcommand{\re}{\mbox{Re}\,\xspace}
\newcommand{\im}{\mbox{Im}\,\xspace}
\newcommand{\norm}[1]{\left \lVert #1 \right \rVert\xspace}

\begin{document}
	
	\preprint{AIP/123-QED}\
	\title[Geometric invariance of determining and resonating centers:  Odd- and any-number limitations of Pyragas control]
	{Geometric invariance of determining and resonating centers: \\ Odd- and any-number limitations of Pyragas control}
	% Force line breaks with \\

	\author{B. de Wolff}
	\email{bajdewolff@zedat.fu-berlin.de}
	\author{I. Schneider}%
	\email{isabelle.schneider@fu-berlin.de}
	\affiliation{ 
		Freie Universität Berlin, Institut für Mathematik, Arnimallee 7, 14195 Berlin, Germany%\\This line break forced with \textbackslash\textbackslash
	}%
	
	% \homepage{http://www.Second.institution.edu/~Charlie.Author.}
	%\affiliation{%
	%Second institution and/or address%\\This line break forced% with \\
	%}%
	
	\date{\today}% It is always \today, today,
	
	\begin{abstract}
		In the spirit of the well-known odd-number limitation, we study failure of Pyragas control of periodic orbits and equilibria. Addressing the periodic orbits first, we derive a fundamental observation on the invariance of the geometric multiplicity of the trivial Floquet multiplier. This observation leads to a clear and unifying understanding of the odd-number limitation, both in the autonomous and the non-autonomous setting. Since the presence of the trivial Floquet multiplier governs the possibility of successful stabilization, we refer to this multiplier as the determining center. The geometric invariance of the determining center also leads to a necessary condition on the gain matrix for the control to be successful. In particular, we exclude scalar gains.
		Application of Pyragas control on  equilibria does not only imply a geometric invariance of the determining center, but surprisingly also on centers which resonate with the time delay. Consequently, we formulate odd- and any-number limitations  both for real eigenvalues together with arbitrary time delay as well as for complex conjugated eigenvalue pairs together with a resonating time delay.
		The very general nature of our results allows for various applications.
	\end{abstract}

	\maketitle

	\begin{quotation}
		Time-delayed feedback control using the Pyragas method is an important tool for the stabilization of equilibria and periodic orbits. 
		However, the time delay generates an infinite dimensional system and general results on the controllability are hard to obtain.
		With the present paper, we clarify the confusion in the literature on the odd-number limitation.
		In order to obtain our results, we shift the emphasis from analytic to geometric aspects, and concentrate on properties of the unstable object itself rather than on its dynamical system.
		In this way, we obtain fundamental geometric invariance principles from which all further theorems are deduced.
	\end{quotation}

	\section{\label{sec:introduction}Introduction}
	
	In a dynamical system given by the ordinary differential equation $\dot{x}(t)=f(x(t))$, $x \in \mathbb{R}^{N}$, unstable periodic orbits can be stabilized using additive control terms of the form
	\begin{equation}\label{pyragas}
		K\big(x(t)-x(t-T)\big).
	\end{equation}
	Here $T>0$ is the time delay, and $K \in \mathbb{R}^{N\times N}$ is the weight of the control term, which we call the \emph{gain matrix}. 
	Such control terms were first introduced by Kestutis Pyragas in his work from 1992 \cite{PYR92}.
	The control term by Pyragas uses the difference between the delayed state $x(t-T)$ and the current state $x(t)$ of the system.
	Frequently, the time delay $T$ is chosen to be an integer multiple of the period of the periodic orbit $x_\ast(t)$ of the uncontrolled system.
	In this case, the control vanishes on the orbit itself, and $x_\ast(t)$ is also a solution of the controlled system.
	We call such a control term \emph{noninvasive} because it does not change the periodic orbit itself, but only affects its stability properties.
	In the case of equilibria, the time delay $T$ can be chosen arbitrarily to achieve noninvasiveness.
	
	The main advantage of the Pyragas control scheme, also for experimental realizations, is its model-independence; no expensive calculations are needed for its implementation, and the only information needed is the period of the targeted periodic orbit.
	As a consequence, Pyragas control has many different successful applications, e.g., in atomic force microscopes \cite{YAM09}, un-manned helicopters \cite{OMA12}, complex robots \cite{STE10}, semiconductor lasers \cite{SCH06,SCH11a}, and the enzymatic peroxidase-oxidase reaction \cite{LEK95}, among others. 
	The success of Pyragas control has been verified for a large number of specific theoretical models as well, including spiral break-up in cardiac tissues \cite{RAP99}, flow alignment in sheared liquid
	crystals \cite{STR13},  near Hopf bifurcation \cite{FIE07} and unstable foci \cite{HOE05,YAN06},  synchrony in networks of coupled Stuart-Landau oscillators \cite{SCH13,SCH16}, delay equations \cite{FIE15,FIE17}, in quantum systems \cite{HEI15,DRO19}, the Duffing oscillator \cite{FIE20} and Turing patterns \cite{KUS18}.
	
	General conditions on the success or failure of Pyragas control are hard to obtain because the time delay adds infinitely many  dimensions to the complexity of the dynamical system.
	In fact, there has been serious confusion in the literature on the  so-called \emph{``odd-number limitation''}, which was correctly proven for non-autonomous systems in 1997 \cite{NAK97}; see also Corollary \ref{thm: nakajima} in Section \ref{sec: odd number}. The odd-number limitation states that in non-autonomous periodic ordinary differential equations, hyperbolic periodic orbits with an odd number of real Floquet multipliers larger than one cannot be stabilized using Pyragas control.
	In a footnote, Nakajima formulated the conjecture that the odd-number limitation also holds in the autonomous case and this was subsequently often wrongly cited as a proven fact.
	However, in 2007 Fiedler et al.~found a counter-example: It is possible to stabilize a periodic orbit near a subcritical Hopf bifurcation with one real positive Floquet multiplier \cite{FIE07}. 
	A correct version of the odd-number limitation for autonomous equations was subsequentially presented by Hooton and Amann in 2012 \cite{HOO12}.
	
	In the present paper we  clarify the confusion on the limitations of Pyragas control.
	In contrast to previous works, we focus on the \emph{geometric}, rather than the algebraic, multiplicity of Floquet multipliers (for periodic orbits, see Section II A for a precise definition) and eigenvalues (for equilibria).  
	That is, our interest lies in the dimension of the eigenspace and not on the number of solutions of the characteristic equation.
	Our main results Theorems \ref{lem: preservation} and \ref{lem: preservation eq}  show that the geometric multiplicity of the Floquet multiplier 1, or eigenvalue zero, is invariant under control. Whether such a Floquet multiplier 1, or eigenvalue 0, is present in the uncontrolled system or not decides whether the periodic orbit can in principle be stabilized via Pyragas control. Therefore we refer to geometric eigenspace of the Floquet multiplier 1, or the eigenvalue 0, as the \emph{determining center}.
	
	From the main results Theorem \ref{lem: preservation} and Theorem \ref{lem: preservation eq} we obtain several corollaries, among others the odd-number limitation and an any-number limitation for commuting control matrices.
	Moreover, and rather surprisingly, for steady states, the geometric multiplicities of the \emph{resonating centers} $2 \pi n i/T$, with $n \in \mathbb{Z}$ and $T$ the time delay, are also preserved under control. As a corollary, we obtain that not only real eigenvalues are impossible to stabilize using commuting gain matrices, but also those of the form $\lambda\pm 2 \pi i n /T$. 
	All results are of a qualitative nature, and apply to any ordinary differential equation (ODE) subject to Pyragas control.  They do not give any quantitative restrictions on the Floquet multipliers \cite{FIE08,JUS99} or on the time delay \cite{YAN06}. 
	
	This paper is organized as follows:  
	In Section II, we investigate Pyragas control of periodic orbits. 
	We formulate the fundamental principle of the invariance of the geometric multiplicity of the trivial, yet determining Floquet multiplier 1.
	As corollaries, we prove the odd-number limitation as well as any-number limitations for commuting gain matrices, both with real and complex spectrum. 
	In Section III, we study Pyragas control of equilibria. 
	The main invariance principle here concerns the determining as well as resonating centers.
	From this, we derive an odd-number limitation for equilibria as well as any-number limitations for commuting gain matrices with either real or complex spectrum. We conclude the paper with a discussion on applications and generalizations in Section IV.

	%%%%%%%%%%%%%%%%%%%%%%%%%%%%%%%%%%%%%%%%%%%%%%%%%%%%%%%%%%%%%%%%%%%%%%%%%%%%%%%%
	\section{Geometric invariance of the determining center for periodic orbits}
	
	In this section, we focus on feedback stabilization of periodic orbits. The main result concerns invariance of the geometric multiplicity of the Floquet multiplier $1$ under Pyragas control (Theorem \ref{lem: preservation}). From this invariance we deduce several limitations on feedback stabilization. Since the presence of the Floquet multiplier $1$ determines whether a periodic solution can in principle be stabilized, we refer to it as the \emph{determining center}. 
	
	The geometric invariance of the determining center is all the more striking since we compare a center in a finite-dimensional system, given by eigenvectors, to a center in an infinite-dimensional system, given by eigenfunctions. 
	Still, there is a one-to-one correspondence between these eigenvectors and eigenfunctions, 
	which we explain in the following.
	
	\subsection{Main result concerning periodic orbits } \label{sec:invariance periodic}
	Throughout we consider the ODE
	\begin{align} \label{eq: time periodic ode}
		\dot{x}(t) = f(x(t), t), \qquad t \geq 0, 
	\end{align}
	with $f: \mathbb{R}^N \times \mathbb{R} \to \mathbb{R}^N$ a $C^1$-function. 
	We make the following, very general, assumptions on system \eqref{eq: time periodic ode}:
	\begin{assumption} \hfill \label{assumption}
		\begin{enumerate}
			\item  The function $f$ is periodic with (not necessarily minimal) period $T > 0$ in its time-argument, i.e., $f(x, t+T) = f(x, t)$ for all $x \in \mathbb{R}^{N}$ and $t \in \mathbb{R}$;
			\item System \eqref{eq: time periodic ode} has a periodic solution $x_\ast(t)$ with (again not necessarily minimal) period $T$. 
		\end{enumerate}
	\end{assumption}
	
	Note that besides periodic non-autonomous ODE, Assumption 1 also includes autonomous ODE which possess a periodic orbit of period $T$. In this case, Assumption 1.1 is trivially fulfilled. Indeed, our results in the rest of the section give a unifying approach to both the autonomous and the non-autonomous case. 
	
	Linearizing around the periodic orbit $x_\ast(t)$, we obtain the system 
	\begin{equation} \label{eq: linvareq ode 1}
		\dot{y}(t) = \partial_x f (x_\ast(t), t) y(t).
	\end{equation}
	By $Y_0(t) \in \mathbb{R}^{N \times N}, \ t \geq 0$ we denote the fundamental solution of the linear matrix differential equation:
	\begin{align}
		\begin{cases}
			\frac{d}{dt} Y_0(t) &= \partial_x f(x_\ast(t), t) Y_0(t), \qquad t > 0; \\
			Y_0(0) &= I,
		\end{cases}
	\end{align}
	where $I: \mathbb{R}^N \to \mathbb{R}^N$ denotes the identity matrix. If $y_0 \in \mathbb{R}^N$, then $y(t) : = Y_0(t) y_0$ solves \eqref{eq: linvareq ode 1} with initial condition $y(0)= y_0$, and hence the fundamental solution can be viewed as `summarizing' the solution information to \eqref{eq: linvareq ode 1}. 
	
	We refer to the matrix $Y_0(T): \mathbb{R}^N \to \mathbb{R}^N$ as the \textbf{monodromy operator} of system \eqref{eq: linvareq ode 1}. We define the \textbf{Floquet multipliers} of system \eqref{eq: linvareq ode 1} as the eigenvalues of the monodromy operator $Y_0(T)$. 
	
	Let $\mu \in \mathbb{C}$ be an eigenvalue of $Y_0(T)$, i.e., $\mu$ is a Floquet multiplier of \eqref{eq: linvareq ode 1}. Then the linear space
	\begin{equation} \label{eq: geometric mp}
		\mathcal{N}\left(\mu I - Y_0(T) \right) : = \{ x \in \mathbb{C}^N \mid \mu x - Y_0(T) x = 0 \}
	\end{equation}
	has dimension at least $1$; and moreover, $\mu$ satisfies
	\begin{equation} \label{eq: ce ode}
		0 = \det \left( \mu I - Y_0(T) \right).
	\end{equation}
	The \textbf{geometric multiplicity} of $\mu$ is defined as the dimension of the linear space \eqref{eq: geometric mp} and the \textbf{algebraic multiplicity} is defined as the order of $\mu$ as a zero of the function $z \mapsto \det\left( z I - Y_0(T) \right)$. The geometric and algebraic multiplicity of a Floquet multiplier can in general be different; however, the algebraic multiplicity is always larger than or equal to the geometric multiplicity. In Theorem \ref{lem: preservation}, we show that the geometric multiplicity of the Floquet multiplier $1$ plays a fundamental role in stabilization. Therefore, we will refer to the geometric eigenspace of the Floquet multiplier $1$, i.e., to the space
	\begin{equation}
		\mathcal{N}\left(1 - Y_0(T) \right)
	\end{equation}
	as the \textbf{determining center}. In this definition, we do \emph{not} assume that the space $\mathcal{N}\left(1 - Y_0(T) \right)$ contains more than the zero vector (i.e. we do \emph{not} assume that $1$ is a Floquet multiplier). If indeed the space $\mathcal{N}\left(1 - Y_0(T) \right)$ is non-trivial, it is a subspace of the center \emph{eigenspace}, but instead we refer to it as `center' for brevity of notation.

	We now apply Pyragas control to the system \eqref{eq: time periodic ode} and write the controlled system as
	\begin{equation} \label{eq: time periodic pyragas}
		\dot{x}(t) = f(x(t), t) + K \left[x(t) - x(t-T) \right]
	\end{equation}
	with nonzero gain matrix $K \in \mathbb{R}^{N \times N}$. For $t \geq 0$, denote by 
	\begin{equation}
		Y_1(t): C \left([-T, 0], \mathbb{R}^N\right) \to C \left([-T, 0], \mathbb{R}^N\right)
	\end{equation}
	the fundamental solution of the linearized equation
	\begin{equation} \label{eq: linvareq dde}
		\dot{y}(t) = \partial_x f(x_\ast(t), t) y(t)+ K \left[y(t) - y(t-T) \right].
	\end{equation}
	The map
	\begin{equation}
		Y_1(T): C \left([-T, 0], \mathbb{R}^N \right) \to C \left([-T, 0], \mathbb{R}^N \right)  \end{equation}
	is a bounded linear operator, which is also compact (see appendix for a proof). Compactness of the operator implies that all non-zero spectral points are eigenvalues of finite algebraic multiplicity. In particular, if $\mu \neq 0$ is an eigenvalue of $Y_1(T)$, then the linear space
	\begin{equation} \label{eq: geom mp dde}
		\mathcal{N}\left( \mu I - Y_1(T) \right) : = \{ \phi \in  C \left([-T, 0], \mathbb{C}^N \right) \mid \mu \phi - Y_1(T) \phi = 0 \}
	\end{equation}
	is finite dimensional and the geometric multiplicity of $\mu$ equals the dimension of the space \eqref{eq: geom mp dde}. The novelty of our approach is that we initially focus on the geometric, rather than the algebraic multiplicity, of the Floquet multipliers. We show that the geometric multiplicity of the Floquet multiplier $1$ is preserved under control. This then serves as an determining principle which decides whether the targeted periodic solution can, in principle, be stabilized.

	The following main result compares the geometric multiplicity of the eigenvalue $1$ of $Y_0(T)$ with the geometric multiplicity of the eigenvalue $1$ of $Y_1(T)$. By convention, if $1$ is \emph{not} an eigenvalue of $Y_0(T)$ (resp., $Y_1(T)$), we say that the geometric multiplicity of the eigenvalue $1$ of $Y_0(T)$ (resp., $Y_1(T)$) is zero. 
	
	\begin{theorem}[Geometric invariance of the determining center under Pyragas control] \label{lem: preservation}
		The geometric multiplicity of the Floquet multiplier $1$ is preserved under Pyragas control.
		That is, for any gain matrix $K \in \mathbb{R}^{N\times N}$, the geometric multiplicity of the eigenvalue 1 of $Y_0(T)$ without control is equal to the geometric multiplicity of the eigenvalue 1 of $Y_1(T)$ with control. 
	\end{theorem}
	
	\begin{proof}
		We show that there is a one-to-one correspondence between eigenvectors to the eigenvalue $1$ for $Y_0(T)$ and eigenfunctions to the eigenvalue $1$ of $Y_1(T)$. The statement of the claim then follows.
		
		On the one hand, the vector $y_0 \in \mathbb{C}^{N} \backslash \{0 \}$ is an eigenvector of $Y_0(T)$ with eigenvalue $1 \in \mathbb{C}$ if and only if \eqref{eq: linvareq ode 1} has a solution $y(t)$ that satisfies
		\begin{align} \label{eq: eigenvector ode}
			\begin{cases}
				y(t+T) =  y(t), \quad t \in \mathbb{R} \\
				y(0) = y_0;
			\end{cases}
		\end{align}
		see also Appendix \ref{sec:appendix ode}. On the other hand, $\phi \in C \left([-T,0], \mathbb{C}^N \right) \backslash \{0 \}$ is an eigenfunction of $Y_1(T)$ with eigenvalue $1 \in \mathbb{C}$ if and only if \eqref{eq: linvareq dde} has a solution $y(t)$ that satisfies
		\begin{align}  \label{eq: eigenvector dde}
			\begin{cases}
				y(t+T) =  y(t), \quad t \geq 0 \\
				y(t) = \phi(t), \quad t \in [-T, 0];
			\end{cases}
		\end{align}
		see also Appendix \ref{sec:appendix dde}. But since the control term $K \left[y(t) - y(t-T) \right]$ vanishes on $T$-periodic functions, \eqref{eq: linvareq dde} has a solution of the form \eqref{eq: eigenvector dde} \emph{if and only if} \eqref{eq: linvareq ode 1} has a solution of the form \eqref{eq: eigenvector ode}.
		We conclude that there is a one-to-one correspondence between eigenvectors of the eigenvalue $1$ of $Y_0(T)$ and eigenfunctions of the eigenvalue 1 of $Y_1(T)$, which proves the claim. 
	\end{proof}
	
	%%%%%%%%%%%%%%%%%%%%%%%%%%%%%%%%%%%%%%%%%%%%%%%%%%%%%%%%%%%%
	\subsection{Corollary: The odd-number limitation} \label{sec: odd number}
	
	Strictly speaking, our main result does not make any statement about stabilization via Pyragas control, only addressing the seemingly unimportant center.
	However, we can use it to deduce a number of restrictions, that is, necessary conditions, on Pyragas control. 
	We start with the long-known odd-number limitation, which follows easily and  can now be fully understood in this context of geometric multiplicities.
	
	Previous statements of the odd-number limitation are formulated for either autonomous or non-autonomous systems. In contrast, we formulate the odd-number limitation for \emph{non-degenerate periodic orbits}, i.e., periodic orbits that do not have a Floquet multiplier $1$. By shifting the focus from (non)-autonomous systems to (non)-degenerate periodic orbits, we clarify the confusion in the literature regarding the odd-number limitation.
	
	For non-degenerate periodic orbits, the absence of a Floquet multiplier $1$ in the uncontrolled system forbids stabilization, as no other path for real eigenvalues is possible. 
	In autonomous systems, every periodic orbit (provided it is not an equilibrium) is degenerate, since translation along the periodic orbit leads to a trivial Floquet multiplier. A clever choice of the gain matrix can allow for a change of the algebraic multiplicity while leaving the geometric multiplicity invariant, thus achieving stabilization.

	As a technical prerequisite, we first state and prove the odd-number limitation on the linear level. This has the advantage that, in blockdiagonalizing ODE, the linear statement can be applied  to individual blocks.
	
	\begin{prop} \label{prop: odd number technical}
		Consider the linear system
		\begin{equation} \label{eq: A}
			\dot{y}(t) = A(t) y(t)
		\end{equation}
		with $y(t) \in \mathbb{R}^N$ and $A(t) \in \mathbb{R}^{N \times N}$. Assume that there exists a time $T > 0$ such that $A(t+T) = A(t)$ for all $t \in \mathbb{R}$. Moreover, assume that system \eqref{eq: A} does not have a Floquet multiplier equal to 1 and that it possesses an odd number (counting algebraic multiplicities) of real Floquet multipliers strictly larger than 1. 
		
		Then, for all gain matrices $K \in \mathbb{R}^{N \times N}$, the controlled system
		\begin{equation} \label{eq: A control}
			\dot{y}(t) = A(t) y(t) + K \left[y(t) - y(t-T) \right]
		\end{equation}
		has at least one real Floquet multiplier larger than 1. 
	\end{prop}
	
	\begin{proof}
		We first give an intuitive argument: By assumption, the linearized system \eqref{eq: A} has an odd number of real Floquet multipliers on the line $(1, \infty)$. % but none in the determining center. 
		But, since we use a real control matrix $K \in \mathbb{R}^{N \times N}$, non-real Floquet multipliers of the controlled system appear in complex conjugated pairs. Therefore, it is impossible to change the parity of eigenvalues on the line $(1, \infty)$ by leaving the real axis and at least one Floquet multiplier stays on the line $(1, \infty)$. This Floquet multiplier can only move into the unit circle by crossing the point $1 \in \mathbb{C}$. However, this is forbidden by the invariance of the determining center (Theorem \ref{lem: preservation}).

		To make this argument precise, fix a gain matrix $K \in \mathbb{R}^{N \times N}$ and introduce the homotopy parameter $\alpha \in [0,1]$:
		\begin{equation}\label{eq: parameter}
			\dot{y}(t) = A(t) y(t) + \alpha K \left[y(t) - y(t-T) \right].
		\end{equation}
		The monodromy operator $Y_\alpha(T)$ for \eqref{eq: parameter} is compact for all $\alpha \in [0,1]$, and the map $\alpha \mapsto Y_\alpha(T)$ is continuous; therefore the Floquet multpliers of \eqref{eq: parameter} (i.e., the eigenvalues of $Y_\alpha(T)$) depend continuously on $\alpha$ (in the sense of \cite{Kato95}).
		
		We first show that the number of Floquet multipliers outside the unit circle cannot change by multipliers ``coming from infinity''. Indeed, let $\mu$ be an eigenvalue of $Y_\alpha(T)$ and let $\phi$ be such that
		\begin{equation}
			Y_\alpha(T) \phi = \mu \phi, \qquad \norm{\phi} = 1,
		\end{equation}
		i.e., $\phi$ is an eigenfunction with norm equal to $1$. Then 
		\begin{equation}
			\left| \mu \right| = \norm{\mu \phi} = \norm{Y_\alpha(T) \phi} \leq \norm{Y_\alpha(T)}.   
		\end{equation}
		Thus, if $\mu$ is an eigenvalue of $Y_\alpha(T)$, then $\left| \mu \right| \leq \norm{Y_\alpha(T)}$, i.e., the norm of $\mu$ can be bounded by the operator norm of $Y_\alpha(T)$. For every $\alpha \in [0,1]$, the operator $Y_\alpha(T)$ is bounded and the map $ \alpha \mapsto Y_\alpha(T)$ is continuous. Therefore, there exists a $0 < C < \infty$ such that 
		\begin{equation}
			\sup_{\alpha \in [0,1]} \norm{Y_\alpha(T)} < C. 
		\end{equation}
		We conclude that if $\mu$ is an eigenvalue of $Y_\alpha(T)$, then $\left| \mu \right| < C$. Hence the number of eigenvalues of $Y_\alpha(T)$ that lies outside the unit circle cannot change by an eigenvalue ``coming from infinity''; the number of eigenvalues of $Y_\alpha(T)$ outside the unit circle can only change by an eigenvalue crossing the unit circle.

		Now, for $\alpha \in [0,1]$, define 
		\begin{equation}
			n_\alpha = \# \{ \mu \in \sigma_{pt}(Y_\alpha(T)) \mid \mu \in (1, \infty) \}   
		\end{equation}
		i.e., $n_\alpha$ is the number of eigenvalues of $Y_\alpha(T)$ lying on the half-line $(1, \infty)$. 
		%By assumption, $n_{\alpha = 0}$ is odd.
		Note that the parity of $n_\alpha$ can only change by an eigenvalue crossing the point $1 \in \mathbb{C}$: 
		Indeed, if $\mu$ is an eigenvalue of $Y_\alpha(T)$, then $\overline{\mu}$ is an eigenvalue as well, so non-real eigenvalues appear in pairs which do not affect the parity of $n_\alpha$.
		
		The only way left is through the real line, i.e., through $1 \in \mathbb{C}$.
		However, by assumption, $1 \in \mathbb{C}$ is not a Floquet multiplier of \eqref{eq: parameter} for $\alpha = 0 $ and, by Theorem \ref{lem: preservation}, it will not be a Floquet multiplier of \eqref{eq: parameter} for any $\alpha \in (0, 1]$. 
		Hence it is impossible to change the parity of $n_\alpha$ through Pyragas control. Since by assumption, $n_{\alpha = 0}$ is odd, we conclude that $n_{\alpha}$ is odd for all $\alpha \in [0,1]$, and \eqref{eq: A control} has at least 1 Floquet multiplier larger than one. 
	\end{proof}
	
	Note how the proof combines continuous dependence on parameters of the Floquet multipliers with the geometric invariance of the center to conclude that stabilization is impossible. 
	The assumption that $1 \in \mathbb{C}$ is not a multiplier of the uncontrolled system is crucial here. To facilitate terminology on this essential assumption, we say that $x_\ast$ is \emph{non-degenerate} as a solution of \eqref{eq: time periodic ode} if the linearization \eqref{eq: linvareq ode 1} does \emph{not} have a Floquet multiplier 1. 
	
	For non-degenerate periodic solutions of non-autonomous ODE, we recover the well-known odd-number limitation \cite{NAK97}:
	
	\begin{cor}[Nakajima\cite{NAK97}, '97] \label{thm: nakajima} Consider the system \eqref{eq: time periodic ode} satisfying Assumption \ref{assumption}. Assume that $x_\ast$ is non-degenerate as a solution of \eqref{eq: time periodic ode} and that the linearized equation \eqref{eq: linvareq ode 1} has an odd number (counting algebraic multiplicities) of real Floquet multipliers larger than $1$. 
		
		Then, for every gain matrix $K \in \mathbb{R}^{N \times N}$, the periodic solution $x_\ast$ is unstable as a solution of the controlled system \eqref{eq: time periodic pyragas}. 
	\end{cor}
	\begin{proof}
		We apply Proposition \ref{prop: odd number technical} with $A(t) = \partial_x f(x_\ast(t), t)$: By assumption, 
		\begin{equation}
			\dot{y}(t) = \partial_x f(x_\ast(t), t) y(t)    
		\end{equation} 
		does not have a Floquet multiplier 1 and has an odd number of Floquet multipliers larger than 1. Thus Proposition \ref{prop: odd number technical} implies that 
		\begin{equation}
			\dot{y}(t) =  \partial_x f(x_\ast(t), t) y(t) + K \left[y(t) - y(t-T) \right] 
		\end{equation}
		has a Floquet multiplier larger than 1 for any gain $K$. Therefore $x_\ast$ is unstable as a solution of \eqref{eq: time periodic pyragas}. 
	\end{proof}
	
	 For any linear, time-periodic DDE, the eigenvalues of the monodromy operator are also captured by a finite dimensional function called the characteristic matrix function \cite{KaashoekVL92, Sieber11, KaashoekVL20}. For system \eqref{eq: A control}, where the time delay is equal to the period, the expression for this characteristic matrix function is relatively explicit. In fact, the original proof in \cite{NAK97} relies heavily on the explicit form of the characteristic matrix function. This is unnecessary, as the argument here shows: the odd number limitation follows directly from the invariance principle in Theorem \ref{lem: preservation}; it does \emph{not} rely on the fact that in system \eqref{eq: A control} the delay is equal to the period.

	%%%%%%%%%%%%%%%%%%%%%%%%%%%%%%%%%%%%%%%%%%%%%%%%%%%%%%%%%%%%%%%%%%%%%%%%%%%%%%%
	Let us shortly reflect the situation for autonomous systems \eqref{eq: time periodic ode}, i.e., if
	\begin{equation} \label{eq: autonomous}
		\partial_t f(x, t) = 0 \qquad \mbox{for all } x \in \mathbb{R}^n \mbox{ and } t \in \mathbb{R}.
	\end{equation}
	In this case, if the periodic orbit $x_\ast$ is not an equilibrium, we differentiate the relation
	\begin{equation}
		\dot{x}_\ast(t) = f(x_\ast(t), t)
	\end{equation}
	with respect to $t$ to see that $\dot{x}_\ast(t)$ is a non-zero, $T$-periodic solution of \eqref{eq: linvareq ode 1}. 
	It follows that \eqref{eq: linvareq ode 1} has a Floquet multiplier $1$ (called the \emph{trivial Floquet multiplier}). %and the center is nonempty. %and its geometric multiplicity is preserved under control. 
	Thus, if system \eqref{eq: time periodic ode} is in fact autonomous, the solution $x_\ast$ is degenerate and the assumptions of Corollary \ref{thm: nakajima} are not satisfied. 
	
	Moreover, if \eqref{eq: time periodic ode} is autonomous, Theorem \ref{lem: preservation} implies that the geometric multiplicity of the trivial Floquet multiplier is preserved.
	However, its  \emph{algebraic multiplicity} is not fixed under control, and changing the algebraic multiplicity is necessary for successful stabilization. Indeed, the results from Hooton and Amann \cite{HOO12} on stabilization in autonomous systems have a natural interpretation in terms of the algebraic multiplicity of the trivial Floquet multiplier (as will be discussed in more detail upcoming work by the first author \cite{deWolff21}). Also the positive stabilization result \cite{FIE07} for an autonomous system shows stabilization through the center generated by the trivial Floquet multiplier 1.
	
	Some of the results  up to this point -- most notably, the invariance principle and the odd-number limitation -- \emph{might} (!) apply to a more general class of noninvasive control terms than `only' Pyragas control. 
		However, before drawing conclusions on different control terms, one should carefully consider the functional analytical framework, in particular, whether the monodromy operator is still a Riesz operator and whether eigenvalues depend continuously on parameters.

	%%%%%%%%%%%%%%%%%%%%%%%%%%%%%%%%%%%%%%%%%%%%%%%%%%%%%%%%%%%%%%%
	\subsection{Corollary: Any-number limitation for commuting gain matrices with real spectrum} \label{sec:real periodic}

	In addition to the odd-number limitation, we  obtain direct restrictions on the choice of the gain matrix, again directly from the geometric invariance of the determining center.
	This will be explored in this subsection:
	In summary, in combination with real Floquet multipliers, stabilization is impossible if the gain matrix commutes with the linearization. 
	In particular, scalar gains are excluded.
	This statement is independent on the actual number of real Floquet multipliers larger than 1, therefore we call it the \emph{any-number limitation for commuting gain matrices}.
	
	We first formulate our results on a linear level. We use Floquet theory (see also Appendix \ref{sec:appendix ode}--\ref{sec:appendix dde}) to transform the linear, time-periodic system  $y(t) = A(t) y(t)$ into an autonomous one.
	Assume that the map $t \mapsto A(t)$ is $T$-periodic; denote by $Y_0(T)$ its monodromy operator. 
	Since $Y_0(T)$ is invertible (i.e., $0 \not \in \sigma(Y_0(T, 0))$), there exists a matrix $B \in \mathbb{C}^{N \times N}$ such that
	\begin{subequations}
		\begin{equation} \label{eq: B defn}
			Y_0(T) = e^{BT}.
		\end{equation}
		Floquet theory for ODE (see Appendix \ref{sec:appendix ode})  gives that the map
		\begin{equation} \label{eq: P defn}
			P(t): = Y_0(t) e^{- Bt}
		\end{equation}
	\end{subequations}
	is $T$-periodic; moreover, the coordinate transformation $y(t) = P(t) v(t)$ transforms the time-periodic system $\dot{y}(t)=A(t)y(t)$ into the linear, autonomous system
	\begin{equation} \label{eq: B ode}
		\dot{v}(t) = B v(t).
	\end{equation}
	In the next proposition, we consider gain matrices $K$ that are \emph{commutative} in the sense that 
	\begin{equation}\label{commutative}
		K P(t) = P(t) K \quad \& \quad K B = BK 
	\end{equation}
	for all $t > 0$.
	This seemingly restrictive assumption can easily be fulfilled, as it is trivially true for scalar gains, or, in the case of symmetric systems, for any gain matrix which leaves the periodic orbit invariant pointwise.
	%with the same symmetry as the original system. 
	In fact, in this case, gain matrices fulfilling assumption \eqref{commutative} seem a very natural choice.
	
	In the next proposition, we provide an analogous statement to Proposition \ref{prop: odd number technical}, but with the above assumptions on the gain matrix.
	Regarding the Floquet theory of the uncontrolled system, we assume that there \emph{exists} a real Floquet multiplier strictly  larger than 1, but we make no assumptions on the number or parity of such multipliers. Moreover, we do not make assumptions on the presence of a Floquet multiplier 1, and thus the result can be applied in both autonomous and non-autonomous settings.

	\begin{prop} \label{prop: even number technical}
		Consider the linear, non-autonomous system 
		\begin{equation} \label{eq: A 2}
			\dot{y}(t) = A(t) y(t)
		\end{equation}
		and assume that there exists $T > 0$ such that $A(t+T) = A(t)$ for all $t \in \mathbb{R}$. Assume that system \eqref{eq: A 2} has at least one real Floquet multiplier larger than $1$. 
		
		Assume that the gain matrix $K \in \mathbb{R}^{N \times N}$ satisfies $\sigma(K) \subseteq \mathbb{R}$. Moreover, with $P(t)$ and $B$ as in \eqref{eq: B defn}--\eqref{eq: P defn}, assume that 
		\begin{subequations}
			\begin{align}
				K P(t) = & \ P(t) K \quad\mbox{ for all }t \in \mathbb{R};\label{eq: commute P} \\
				K B = & \ B K. \label{eq: commute B}
			\end{align}
		\end{subequations}
		Then the controlled system
		\begin{equation} \label{eq: A control 2}
			\dot{y}(t) = A(t) y(t) + K \left[y(t) - y(t-T) \right]
		\end{equation}
		has at least one real Floquet multiplier larger than $1$. 
	\end{prop}
	\begin{proof}
		We divide the proof into three steps: 
		
		\emph{Step 1:} We first transform the DDE \eqref{eq: A control 2} into an autonomous DDE. Recall that the coordinate transform $y(t) =P(t) v(t)$ transforms solutions of \eqref{eq: A 2} into solutions of \eqref{eq: B ode}. Therefore, the coordinate transformation $y(t) = P(t) v(t)$ transforms solutions of \eqref{eq: A control 2} into solutions of 
		\begin{align}
			\dot{v}(t) &= B v(t) + P(t)^{-1} K \left[P(t) v(t) - P(t) v(t-T) \right] \\
			&=B v(t) + K \left[v(t) - v(t-T)\right],
		\end{align}
		where we have used that $P(t+T) = P(t)$ and identity \eqref{eq: commute P}. Thus, the invertible, time-periodic coordinate transformation $y(t) =P(t) v(t)$ transforms solutions of \eqref{eq: A control 2} into solutions of 
		\begin{equation} \label{eq: transformed}
			\dot{v}(t)  = B v(t) + K \left[v(t) - v(t-T)\right].
		\end{equation}
		
		\emph{Step 2}: Next we use the commutativity property \eqref{eq: commute B} to find a common eigenvector for the unstable eigenvalue of the uncontrolled system and the gain matrix. Since by assumption $Y_0(T) = e^{BT}$ has an eigenvalue $\mu_\ast > 1$, we can choose $B$ such that $B$ has an eigenvalue $\lambda_\ast > 0$. Now let $y \in \mathbb{C}^N$ be such that $\lambda_\ast y- B y = 0$. Then
		\begin{equation}
			(\lambda_\ast I -B) Ky = K (\lambda_\ast I - B) y = 0
		\end{equation} 
		since, by \eqref{eq: commute B}, $B$ and $K$ commute. Therefore the space
		\begin{equation} \label{eq:space}
			\mathcal{N}  \left( \lambda_\ast I - B \right) : = \{y \in \mathbb{C}^N \mid (\lambda_\ast I - B) y = 0 \}
		\end{equation} 
		is invariant under $K$. 
		%Since $K$ is by assumption diagonalizable over $\mathbb{C}^N$, it is also diagonalizable over $\mathcal{N} \left( \lambda_\ast I - B \right)$. 
		Hence we can find a non-zero $y_\ast \in\mathcal{N} \left( \lambda_\ast I - B \right)$ and a $k_\ast \in \sigma(K)$ such that $K y_\ast = k_\ast y_\ast$. So we conclude we can find a $y_\ast \in \mathbb{C}^N \backslash \{0 \}$ such that 
		\begin{equation} \label{eq: eigenvector}
			B y_\ast = \lambda_\ast y_\ast, \quad K y_\ast = k_\ast y_
			\ast
		\end{equation}
		i.e., $y_\ast$ is a simultaneous eigenvector for $B$ and $K$.

		\emph{Step 3:} We reduce to a 1-dimensional, real valued DDE using the common eigenvector from Step 2. Consider the real-valued, scalar DDE
		\begin{equation} \label{eq: scalar dde}
			\dot{w}(t) = \lambda_\ast w(t) + k_\ast \left[w(t) - w(t-T)\right].
		\end{equation}
		Since $\lambda_\ast > 0$, the ODE
		\begin{equation}
			\dot{w}(t) = \lambda_\ast w(t) 
		\end{equation} 
		has one Floquet multiplier $e^{\lambda_\ast T}  >1$ and no trivial Floquet multiplier. Therefore, Proposition \ref{prop: odd number technical} implies that \eqref{eq: scalar dde} has at least one trivial Floquet multiplier $\mu$ larger than $1$, i.e. \eqref{eq: scalar dde} has a solution $w_\mu(t)$ with $w_\mu(t+T) = \mu w_\mu(t)$. 
		
		If $w(t) \in \mathbb{R}$ is a solution of \eqref{eq: scalar dde} and $y_\ast$ is as in \eqref{eq: eigenvector}, then $v(t) = w(t) y_\ast$ solves
		\begin{align}
			\dot{v}(t) &= \lambda_\ast w(t) y_\ast + k_\ast \left[w(t) y_\ast - w(t-T) y_\ast \right] \\
			&= B v(t) + K \left[v(t) - v(t-T) \right].
		\end{align}
		Thus, if $w(t)$ solves \eqref{eq: scalar dde}, then $v(t) = w(t) y_\ast$ solves \eqref{eq: transformed}. In particular, $v_\mu(t) : = w_\mu(t) y_\ast$ is a solution of \eqref{eq: transformed} with $v_\mu(t+T) = \mu v_\mu(t)$. This implies that $y_\mu(t) := P(t) v_\mu(t) $ is a solution of \eqref{eq: A control 2} with $y_\mu(t+T) = \mu y_\mu(t)$, i.e., $\mu > 1$ is a Floquet multiplier of \eqref{eq: A control 2}. 
	\end{proof}
	
	Remarkably, Proposition \ref{prop: even number technical} does not make any assumptions on the multiplicity of the unstable Floquet multiplier in the uncontrolled system. Therefore Proposition \ref{prop: even number technical} has a wide applicability (see also Corollary \ref{cor:any number} below). However, if in Proposition \ref{prop: even number technical} we additionally assume that system \eqref{eq: A 2} has a Floquet multiplier larger than 1 \emph{with odd geometric multiplicity}, we can drop the assumption that $K$ has real spectrum. 
	
	Indeed, as in the proof of Proposition \ref{prop: even number technical}, the space \eqref{eq:space} is invariant under $K$. Hence, if this space is odd-dimensional, the real matrix $K$ has at least one real eigenvalue in this space. From there, Step 3 of the proof of Proposition \ref{prop: even number technical} implies instability of the controlled system. 
	\medskip

		If the gain matrix $K$ satisfies the conditions \eqref{eq: commute B}--\eqref{eq: commute P}, then in particular $K A(t) = A(t) K$ for all $t$, i.e. the gain matrix commutes with the linear ODE. Under this weaker assumption, a different proof yields the same any-number limitation, see the upcoming work \cite{deWolff21} by the first author.

	As an application of Proposition \ref{prop: even number technical}, we consider the case $K = k I$ with $k \in \mathbb{R}$. Then the conditions \eqref{eq: commute P}--\eqref{eq: commute B} are trivially satisfied. Therefore Proposition \ref{prop: even number technical} leads to the following corollary: 
	
	\begin{cor}[Any-number limitation for scalar gain matrices] \label{cor:any number}
		Consider the system \eqref{eq: time periodic ode} satisfying Assumption \ref{assumption}. Suppose that the linearized equation \eqref{eq: linvareq ode 1} has at least one real Floquet multiplier larger than $1$. 
		
		Then, for every $k \in \mathbb{R}$, $x_\ast$ is unstable as a solution of the controlled system \begin{equation} \label{eq: scalar control}
			\dot{x}(t) = f(x(t), t) + k \left[x(t) - x(t-T)\right]. 
		\end{equation}
	\end{cor}
	
	\subsection{Corollary: Any-number limitation for commuting gain matrices with complex spectrum} \label{sec:complex periodic} 
	
	In this section, we address more restrictions on the choice of gain matrix. We again consider commutative gain matrices, but in contrast to the results in Section \ref{sec:real periodic}, we do not make any assumptions on the spectrum of the matrix $K$. The main point here is that the limitation on control similar to the formulation in Corollary \ref{prop: even number technical} still holds, but the reasoning behind the limitation is different: the result in this section is no longer a directly corollary of the odd-number limitation, but requires an explicit analysis of the relevant Floquet multipliers.

	\begin{prop} \label{prop:complex gain periodic}
		Consider the linear, non-autonomous system 
		\begin{equation} \label{eq: A 3}
			\dot{y}(t) = A(t) y(t)
		\end{equation}
		and assume that there exists $T > 0$ such that $A(t+T) = A(t)$ for all $t \in \mathbb{R}$. Assume that system \eqref{eq: A 3} has at least one real Floquet multiplier larger than $1$. 
		
		Let $P(t), \ B$ be as in \eqref{eq: B defn}--\eqref{eq: P defn} and assume that the gain matrix $K \in \mathbb{R}^{N \times N}$ satisfies 
		\begin{subequations}
			\begin{align}
				K P(t) = & \ P(t) K \quad\mbox{ for all }t \in \mathbb{R};\label{eq: commute P 2} \\
				K B = & \ B K. \label{eq: commute B 2}
			\end{align}
		\end{subequations}
		Then the system
		\begin{equation} \label{eq: A control 3}
			\dot{y}(t) = A(t) y(t) + K \left[y(t) - y(t-T) \right]
		\end{equation}
		has at least one real Floquet multiplier outside the unit circle.
	\end{prop}
	\begin{proof}
		We divide the proof into five steps. Since the first two steps are identical to the first two steps in the proof of Proposition \ref{prop: even number technical}, we will not repeat them here. 
		
		\emph{Step 1:} Let $y(t)$ be a solution of \eqref{eq: A control 3} and let $P(t),\ B$ be as in \eqref{eq: B defn}--\eqref{eq: P defn}. Then the periodic coordinate transformation $y(t) =P(t) v(t)$ transforms solutions of \eqref{eq: A control 3} into solutions of 
		\begin{align}
			\dot{v}(t) =B v(t) + K \left[v(t) - v(t-T)\right].
		\end{align}
		
		\emph{Step 2}: Since by assumption $Y_0(T) = e^{BT}$ has an eigenvalue $\mu_\ast > 1$, we can choose $B$ such that $B$ has an eigenvalue $\lambda_\ast > 0$. Since by assumption \eqref{eq: commute B 2} the matrices $B$ and $K$ commute, the space 
		\begin{equation}
			\mathcal{N}  \left( \lambda_\ast I - B \right) : = \{y \in \mathbb{C}^N \mid (\lambda_\ast I - B) y = 0 \}
		\end{equation} 
		is invariant under $K$. Therefore we can find a non-zero $y_\ast \in \mathcal{N}  \left( \lambda_\ast I - B \right)$ and a (possibly complex!) $k_\ast \in \mathbb{C}$ such that 
		\begin{equation}
			B y_\ast = \lambda_\ast y_\ast, \quad K y_\ast = k_\ast y_
			\ast
		\end{equation}
		i.e., $y_\ast$ is a simultaneous eigenvector for $B$ and $K$.

		\emph{Step 3:} We now consider the reduced, scalar-valued (possibly complex!) DDE
		\begin{equation} \label{eq: scalar dde 2}
			\dot{w}(t) = \lambda_\ast w(t) + k_\ast \left[w(t) - w(t-T)\right].
		\end{equation}
		The DDE \eqref{eq: scalar dde 2} has a solution of the form $w(t) = e^{\lambda t}$ if and only if $\lambda \in \mathbb{C}$ satsifies
		\begin{equation} \label{eq:ce complex}
			\lambda = \lambda_\ast + k_\ast \left(1-e^{-\lambda T} \right). 
		\end{equation}
		If $\lambda$ satisfies \eqref{eq:ce complex}, then $y(t) = P(t) e^{\lambda t} y_\ast$ is a solution of \eqref{eq: A control 3} with $y(t+T) = e^{\lambda T} y(t)$. Therefore, to prove that system \eqref{eq: A control 3} has a Floquet multiplier outside the unit circle, it suffices to prove that equation \eqref{eq:ce complex} has a solution in the right half of the complex plane. To do so, we distinguish between the case where $k_\ast \in \mathbb{R}$ (Step 4) and the case where $k_\ast \in \mathbb{C} \backslash \mathbb{R}$ (Step 5). 
		
		\emph{Step 4:} For $k_\ast = 0$, equation \eqref{eq: scalar dde 2} becomes
		\begin{equation} \dot{w}(t) = \lambda_\ast w(t) 
		\end{equation}
		which has one Floquet multiplier $e^{\lambda_\ast T} >1$. If we assume that $k_\ast \in \mathbb{R}$, Proposition \ref{prop: even number technical} on real gain matrices implies that \eqref{eq: scalar dde 2} has a Floquet multiplier larger than one for all gains $k_\ast$. This proves the theorem in the case that $k_\ast \in \mathbb{R}$. 
		
		\emph{Step 5:} We now consider the case $k_\ast \in \mathbb{C} \backslash \mathbb{R}$. If $\lambda \in \mathbb{C}$ satisfies \eqref{eq:ce complex}, then $\overline{\lambda}$ satisfies 
		\begin{equation}
			\overline{\lambda} = \lambda_\ast + \overline{k}_\ast \left(1 - e^{- \overline{\lambda} T} \right)
		\end{equation}
		and both $e^{\lambda T}$ and $e^{\overline{\lambda}T}$ are Floquet multipliers of \eqref{eq: A control 3}. Hence, for non-real $k_\ast$, non-real solutions of \eqref{eq:ce complex} in the right half of the complex plane lead to \emph{two} unstable Floquet multipliers of \eqref{eq: A control 3}, see also Figure 1. 
		
		If $\lambda$ is a solution of \eqref{eq:ce complex} on the imaginary axis, then $e^{\lambda T}$ is a Floquet multiplier of \eqref{eq: A control 3} on the unit circle. So to search for stability changes of \eqref{eq: A control 3}, we search for solutions of \eqref{eq:ce complex} on the imaginary axis. Equation \eqref{eq:ce complex} has a solution of the form $\lambda = i \omega$ if and only if $k_\ast$ is on the curve
		\begin{equation}\label{Hopfcurves}
			k_\ast(\omega)=\frac{i \omega-\lambda_\ast}{1-e^{-i \omega T}}.
		\end{equation}
		
		\begin{figure} \label{fig:hopfcurves}
			\includegraphics[width=\columnwidth]{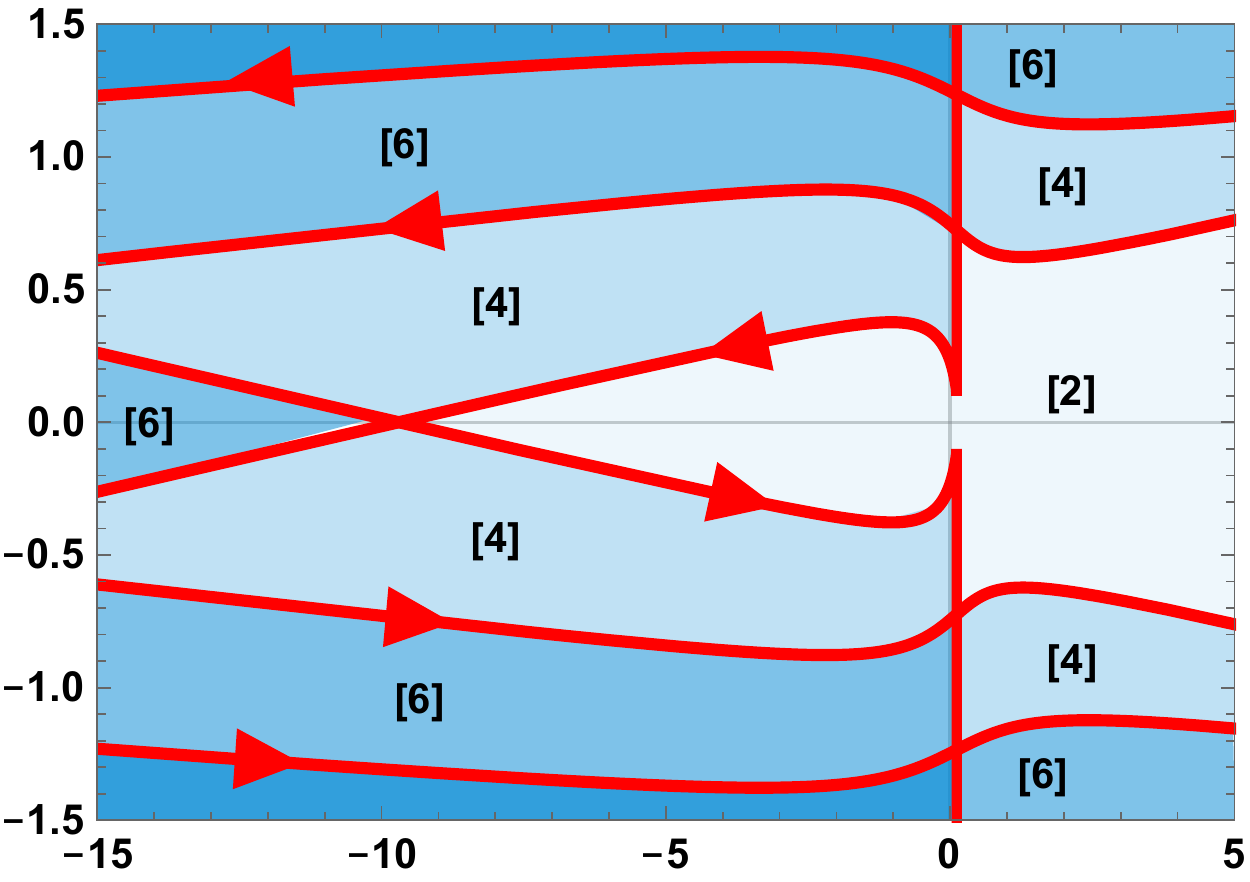}% Here is how to import EPS art
			\caption{\label{Hopfcurvesa005} Hopf curves (red) from eq. \eqref{Hopfcurves}, horizontal axis: Re$(k_\ast)$, vertical axis: Im$(k_\ast)$. Increasing $\omega$ is indicated by arrows. The unstable dimension is given in brackets and is higher by two to the right of the curves. No control is possible. Parameter values: $\lambda_\ast=0.05, T= 2 \pi$.}
		\end{figure}
		\begin{figure}
			\includegraphics[width=\columnwidth]{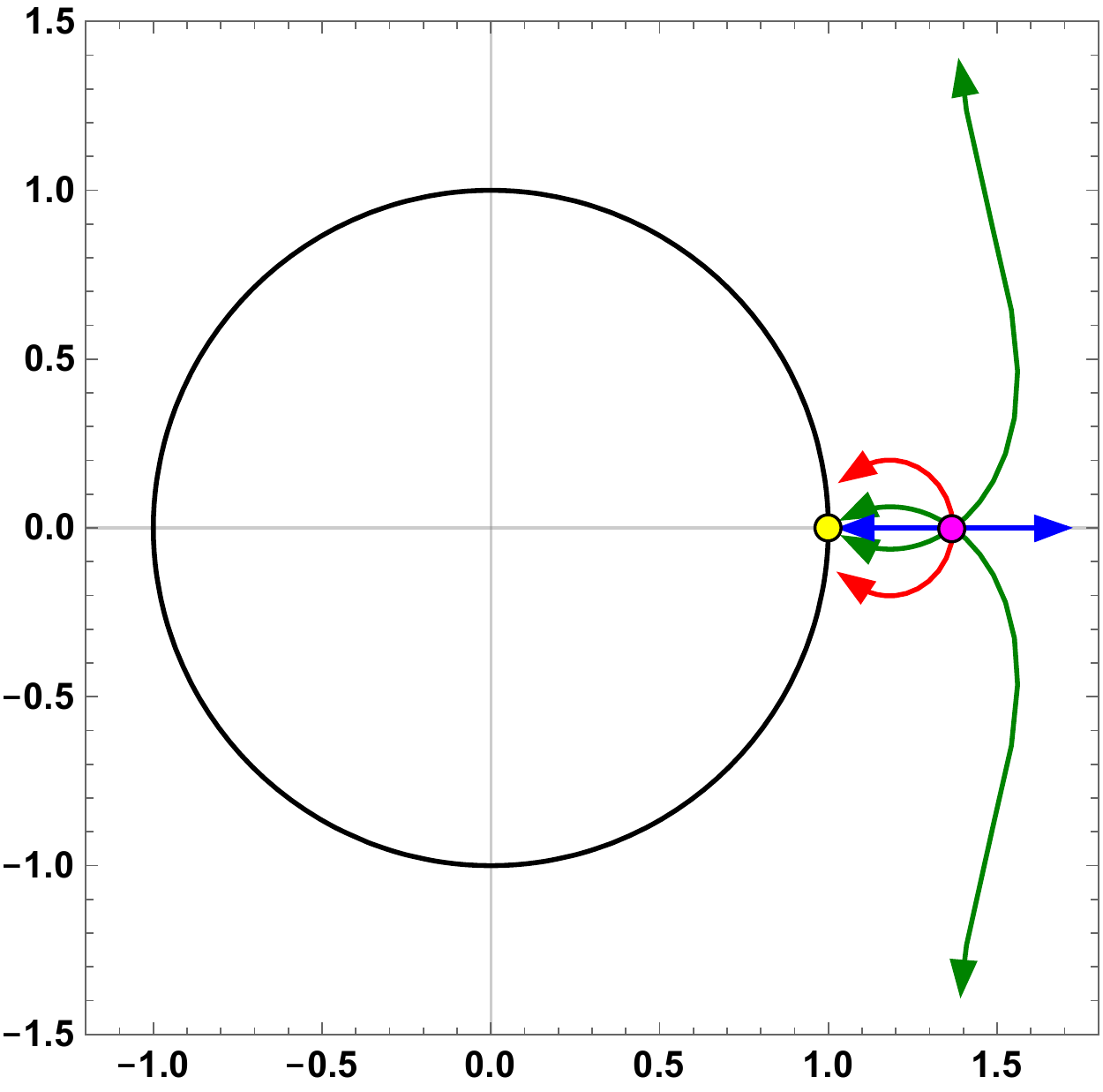}% Here is how to import EPS art
			\caption{\label{circle} Unit circle and Floquet multipliers in the complex plane parametrized by different real and complex $k_\ast$. Yellow dot: 
				Geometrically invariant determining center. Parameter values: $\lambda_\ast=0.05$ (magenta dot), $T= 2 \pi$. Blue: real $k_\ast$ (note the invariance of the real line!), red (dotted): purely imaginary $k_\ast$, green: arg$(k_\ast)=e^{0.8\pi i}$.} 
		\end{figure}
		Note how we can write these curves in one single equation but actually obtain a family of curves through the exponential term in the denominator.
		More precisely, there is one curve for each $\omega \in (2 \pi m/T,2 \pi (m+1)/T)$, $m \in \mathbb{Z}$. 
		Each of these curves defines a continuous  graph over the real axis:
		Indeed, the real part is strictly increasing along the curves. 
		Moreover, we remark that \eqref{Hopfcurves} is symmetric with respect to the real axis. 
		It suffices therefore to consider $\omega>0$, w.l.o.g.
		Next, notice that the curves, parametrized by $\omega$, are oriented in the direction of decreasing real part of $k_\ast$. 
		Moreover, since all the segments are complex differentiable (see also \cite{SCH16}), complex orientation is preserved and we conclude the number of solutions in the right half plane is higher by 2 to the right of the curve.
		Therefore, the unstable dimension can only increase if any of the curves \eqref{Hopfcurves} are crossed, proving the theorem.
	\end{proof}

		Note that in the above proof, in contrast to Proposition \ref{prop: even number technical}, the real line does not stay invariant; see Figure \ref{circle}. 
	In Section \ref{sec:complex equilibria}, we will give an application of Proposition \ref{prop: any number eq} to the stabilization of equilibria.
	
	%%%%%%%%%%%%%%%%%%%%%%%%%%%%%%%%%%%%%%%%%%%%%%%%%%%%%%%%%%%%%%
	\section{Geometric invariance of resonating centers of equilibria}
	In this section, we focus on Pyragas control of equilibria. The main result concerns invariance of the geometric multiplicity not only of the determining center 0, but also of the resonating centers $\pm 2 \pi i n/ T$ (Theorem \ref{lem: preservation eq}). 
	From this invariance we deduce several limitations on feedback stabilization. These limitations explicitly depend on the time delay. 
	If the time delay resonates with the eigenvalues, stabilization becomes impossible. We expect these results to be particularly important at (equivariant) Hopf bifurcation, where the time delay $T$ is in resonance with the purely imaginary eigenvalues $\pm 2 \pi i / T$  of the equilibrium.

	\subsection{Main result concerning equilibria}
	Throughout this section, consider the autonomous ODE 
	\begin{align} \label{eq: autonomous ode}
		\dot{x}(t) = f(x(t))
	\end{align}
	with $f: \mathbb{R}^N \to \mathbb{R}^N$ a $C^1$-function. We assume that there exists a $x_\ast \in \mathbb{R}^N$ such that $f(x_\ast) = 0$, i.e., $x_\ast$ is an equilibrium of \eqref{eq: autonomous ode}. 
	
	The linearization 
	\begin{equation} \label{eq: linear ode eq}
		\dot{y}(t) = f'(x_\ast) y(t)
	\end{equation}
	has an eigenvalue $\lambda$ (or, more precisely, the generator of the semigroup associated to \eqref{eq: linear ode eq} has an eigenvalue $\lambda$) if and only 
	\begin{equation} \label{eq: ce ode 1}
		\det \left( \lambda I - f'(x_\ast) \right) = 0.
	\end{equation}
	The geometric multiplicity of the eigenvalue $\lambda$ is the dimension of the linear space
	\begin{equation}
		\begin{aligned}
			\mathcal{N}\left( \lambda I - f'(x_\ast) \right) = \{x \in \mathbb{C}^N \mid  \lambda x - f'(x_\ast) x = 0 \}.
		\end{aligned}
	\end{equation}
	and its algebraic multiplicity is given by the order of $\lambda$ as a zero of the function $z \mapsto \det \left(z I - f'(x_\ast) \right)$. 
	
	For a fixed time delay $T > 0$, the geometric multiplicity of the eigenvalues $\frac{2 \pi n i}{T}$ of \eqref{eq: linear ode eq} plays an important role in the stabilization results in the rest of this section. To facilitate this in notation, and to emphasize the connection to the time delay, we define the \textbf{resonating centers} to be the geometric eigenspace of the eigenvalues $\frac{2 \pi ni }{T}$, i.e., the space
	\begin{equation}
		\mathcal{N}\left( \frac{2 \pi n i}{T}I - f'(x_\ast) \right).
	\end{equation}
	
	For $T > 0$, $x_\ast$ is again an equilibrium of the controlled system
	\begin{equation} \label{eq: pyragas fixed point}
		\dot{x}(t) = f(x(t)) + K \left[x(t) - x(t-T)\right]
	\end{equation}
	with gain matrix $K \in \mathbb{R}^{N \times N}$.
	Suprisingly, we can find the relevant eigenvalues from a finite-dimensional characteristic equation. Indeed, $\lambda$ is an eigenvalue of the linearization
	\begin{equation} \label{eq: linear dde eq}
		\dot{y}(t) = f'(x_\ast) y(t) + K \left[y(t) -y(t-T)\right],
	\end{equation}
	(or, more precisely, of the generator of the semigroup associated to \eqref{eq: linear dde eq}) if and only if $\lambda$ satisfies
	\begin{equation}
		\det \left( \lambda I - f'(x_0) - K \left[1-e^{-\lambda T}\right] \right) = 0.
	\end{equation}
	The geometric multiplicity of $\lambda$ equals the dimension of the linear space
	\begin{equation}
		\begin{aligned}
			\mathcal{N}\left( \lambda I - f'(x_\ast) - K \left[1-e^{-\lambda T}\right] \right) \subseteq \mathbb{C}^N
		\end{aligned}
	\end{equation}
	and its algebraic multiplicity is the order of $\lambda$ as a zero of the function 
	\begin{equation} \label{eq: ce dde 2}
		z \mapsto \det \left( z I  - f'(x_\ast) - K \left[1-e^{-z T} \right] \right);
	\end{equation}
	see also \cite[Chapter IV]{DIE95}. 
	
	\medskip
	
	In the following theorem, we compare the geometric multiplicity of the resonating eigenvalues of \eqref{eq: linear ode eq} with the geometric multiplicity of the resonating eigenvalues of \eqref{eq: linear dde eq}. By convention, if $\lambda \in \mathbb{C}$ is \emph{not} an eigenvalue of \eqref{eq: linear ode eq} (resp. \eqref{eq: linear dde eq}), we say that the geometric multiplicity of the eigenvalue $\lambda$ of \eqref{eq: linear ode eq} (resp. \eqref{eq: linear dde eq}) is zero.
	
	\begin{theorem}[Geometric invariance of resonating centers] \label{lem: preservation eq}
		For each $n \in \mathbb{Z}$, the geometric multiplicity of the eigenvalue $\frac{2 \pi i n}{T}$ is preserved under control of Pyragas type. That is, for $K \in \mathbb{R}^{N \times N}$ and for every $n \in \mathbb{Z}$, the geometric multiplicity of $\frac{2 \pi n i}{T}$ as an eigenvalue of the ODE
		\begin{equation} \label{eq: linear ode eq 1}
			\dot{y}(t) = f'(x_\ast) y(t)
		\end{equation}
		equals the geometric multiplicity of $\frac{2 \pi n i}{T}$ as an eigenvalue of the DDE
		\begin{equation} \label{eq: linear dde eq 1}
			\dot{y}(t) = f'(x_\ast) y(t) + K \left[y(t) -y(t-T)\right].
		\end{equation}
	\end{theorem}
	\begin{proof}
		Fix $n \in \mathbb{Z}$. The geometric multiplicity of $\frac{2 \pi i n}{T}$ as an eigenvalue of \eqref{eq: linear ode eq 1} is given by
		\begin{equation} \dim \mathcal{N} \left( \frac{2 \pi i n}{T} I - f'(x_\ast) \right). \end{equation}
		Conversely, the geometric multiplicity of $\frac{2 \pi i n}{T}$ as an eigenvalue of \eqref{eq: linear dde eq 1} is given by
		\begin{align}
			\dim  \mathcal{N} \left( \frac{2 \pi i n}{T} I - f'(x_\ast) - K \left[1 - e^{-\frac{2 \pi i n}{T} T} \right] \right) \\
			= \dim \mathcal{N} \left( \frac{2 \pi i n}{T} I - f'(x_\ast) \right).
		\end{align}
		Note that the delay and the resonant eigenvalue cancel in the exponent, causing the contribution of Pyragas control to vanish.
		We conclude that the geometric multiplicity of $\frac{2 \pi i n}{T}$ as an eigenvalue of \eqref{eq: linear ode eq 1} equals the geometric multiplicity of $\frac{2 \pi i n}{T}$ as an eigenvalue of \eqref{eq: linear dde eq 1}. 
	\end{proof}
	
	We briefly reflect on the connection between the two main results on determining and resonating centers. We can interpret the equilibrium $x_\ast$ of \eqref{eq: autonomous ode} as a periodic solution of arbitrary period $T>0$.
	Theorem \ref{lem: preservation eq} then implies Theorem \ref{lem: preservation} for equilibria of autonomous systems, but Theorem \ref{lem: preservation eq} on resonating centers is \emph{finer} in the following sense: If we view
	\begin{equation} \label{eq:comparison}
		\dot{y}(t) = f'(x_\ast) y(t)
	\end{equation}
	as a linear, time-periodic equation with trivial time-dependence, its monodromy operator is given by
	\begin{equation}
		Y_0(T) = e^{f'(x_\ast) T}. 
	\end{equation} 
	Therefore, $1 \in \sigma(Y_0(T))$, i.e., 1 is a Floquet multiplier,  if and only if there is (at least one) $n \in \mathbb{Z}$ such that $\frac{2 \pi i n}{T}$ as an eigenvalue of \eqref{eq:comparison}. 
	In this case the geometric multiplicity of $1 \in \sigma(Y_0(T))$ is given by
	\begin{equation}
		\sum_{n \in \mathbb{Z}} \dim \mathcal{N} \left( \frac{2 \pi i n}{T} I - f'(x_\ast) \right).  
	\end{equation} 
	Note that all the resonating eigenvalues $\frac{2 \pi n i}{T}$ for the equilibrium together form the determining Floquet multiplier $1$ for the periodic orbit. Theorem \ref{lem: preservation} only gives information on the geometric multiplicity of the determining Floquet multiplier, thereby losing information on the geometric multiplicity of the individual resonating eigenvalues. 
	
	At equivariant bifurcation points, we can expect a higher geometric multiplicity of a resonating center. Note that the invariance principle in Theorem \ref{lem: preservation eq} in this case forbids \emph{asymptotic} stabilization.

	\subsection{Corollary: Odd-number limitation for equilibria}
	
	In this section, we use Theorem \ref{lem: preservation eq} to prove limitations for stabilization of equilibria of autonomous systems. 
	Again, the absence of an determining center forbids stabilization. To facilitate this essential point in notation, we say that the equilibrium $x_\ast$ of the ODE $\dot{x}(t) =f(x(t))$ is \emph{non-degenerate} if $0$ is not an eigenvalue of the Jacobian $f'(x_\ast)$.
	
	The following proposition gives an odd-number limitation for equilibria, analogous to the odd-number limitation for periodic orbits in Corollary \ref{thm: nakajima}.
	
	\begin{cor}[Odd-number limitation for equilibria] \label{prop: odd number equilibrium}
		Consider the system
		\begin{equation} \label{eq: odd number eq}
			\dot{x}(t) = f(x(t)), \qquad t \geq 0
		\end{equation}
		with $f: \mathbb{R}^N \to \mathbb{R}^N$ a $C^1$-function. Suppose that $x_\ast \in \mathbb{R}^N$ is an non-degenerate equilibrium of \eqref{eq: odd number eq}. Moreover, assume that $f'(x_\ast)$ has an odd number of eigenvalues (counting algebraic multiplicities) in the strict right half plane.
		
		Then, for all $K \in \mathbb{R}^{N \times N}$ and all $T > 0$, $x_\ast$ is unstable as a solution of the controlled system
		\begin{equation} \label{eq: odd number eq control}
			\dot{x}(t) = f(x(t)) + K \left[x(t) - x(t-T)\right]. 
		\end{equation}
	\end{cor}
	\begin{proof}
		We again start by giving an intuitive argument. Since the matrix $f'(x_\ast): \mathbb{R}^N \to \mathbb{R}^N$ is real, its non-real eigenvalues appear in pairs. Therefore, if $f'(x_\ast)$ has an odd number of eigenvalues in the right half plane, an odd number of these eigenvalues is real. So the assumptions of the statement imply that $f'(x_\ast)$ has an odd number of eigenvalues on the positive real axis. 
		
		Since the gain matrix $K \in \mathbb{R}^{N \times N}$ is real, non-real eigenvalues of the controlled system appear in pairs as well. Therefore, at least one eigenvalue stays on the positive, real axis as control is applied. This eigenvalue can only move into the left half of the complex plane by crossing the point $0 \in \mathbb{C}$. However, this is forbidden, since by Theorem \ref{lem: preservation eq} the controlled system never has an eigenvalue $0$. 
		
		\medskip
		To make the argument more precise, we fix a matrix $K \in \mathbb{R}^{N \times N}$ and a time delay $T > 0$. Moreover, we introduce the homotopy parameter $\alpha \in [0, 1]$:
		\begin{equation}
			\dot{x}(t) = f(x(t)) + \alpha K \left[x(t) - x(t-T) \right].
		\end{equation}
		We first show that the number of eigenvalues with positive real part can only change via a crossing of the imaginary axis:
		The linearization
		\begin{equation}\label{eq: parameter equilibrium}
			\dot{y}(t) = f'(x_\ast) y(t) + \alpha K \left[y(t) - y(t-T) \right].
		\end{equation}
		has an eigenvalue $\lambda \in \mathbb{C}$ if and only if $\lambda$ is a zero of the equation
		\begin{equation}
			d (\lambda, \alpha) := \det \left( \lambda I - f'(x_\ast) - \alpha K \left[1-e^{- \lambda T}\right] \right).
		\end{equation} 
		Since the function $d(\lambda, \alpha)$ is analytic in $\lambda$ and continuous in $\alpha$, the zeros of $\lambda \mapsto d(\lambda, \alpha)$ depend continuously on $\alpha$ (in the sense of \cite{Kato95}).
		Moreover, we can find a $C > 0$ such that 
		\begin{equation}
			\sup \{ \re \lambda \mid d(\lambda, \alpha) = 0 \} < C 
		\end{equation}  
		for all $\alpha \in [0, 1]$ (see \cite[Section I.4]{DIE95}). 
		Therefore, when varying $\alpha$, we cannot change the number of zeros of $\lambda \mapsto d(\lambda, \alpha)$ in the strict right half plane by roots ``coming from infinity''. Thus the number of roots of $\lambda \mapsto d(\lambda, \alpha)$ in the strict right half plane can only change by a root crossing the imaginary axis. Moreover, if $d(\lambda, \alpha) = 0$, then also $d(\overline{\lambda}, \alpha) = 0$,  so non-real roots appear in complex conjugated pairs. 
		In particular, the \emph{parity} of the number of roots with positive real parts can only change if a single eigenvalue crosses zero.
		
		Now, for $\alpha \in [0,1]$, define 
		\begin{equation} n_\alpha = \# \{ \lambda \mid d(\lambda, \alpha) = 0 \mbox{ and } \re \lambda > 0\} 
		\end{equation}
		i.e., $n_\alpha$ is the number of roots of $\lambda \mapsto d(\lambda, \alpha)$ in the strict right half of the complex plane. 
		By the previous remarks, the parity of $n_\alpha$ can only change if a root of $\lambda \mapsto d(\lambda, \alpha)$ crosses the point $0 \in \mathbb{C}$. However, by assumption, $0 \not \in \sigma(f'(x_\ast))$; so Theorem \ref{lem: preservation eq} implies that 
		\begin{equation}
			d(0, \alpha) \neq 0   
		\end{equation}
		for all $\alpha \in [0, 1]$. 
		It follows that the parity of $n_\alpha$ cannot change. Since by assumption, $n_{\alpha = 0}$ is odd, we conclude that $n_{\alpha}$ is odd for all $\alpha \in [0,1]$. In particular, 
		\begin{equation}
			\dot{y}(t) = f'(x_\ast) y(t) + K \left[y(t) - y(t-T) \right]
		\end{equation}
		has an odd number (and thus at least one) of eigenvalues in the right half of the complex plane. This implies that $x_\ast$ is unstable as a solution of \eqref{eq: odd number eq control}. 
	\end{proof}
	%\textcolor{blue}{Not quite happy with the proof yet. It looks more complicated than it actually is. I tried some explanations but now some things double...}
	
	We briefly compare the definition of a non-degenerate periodic orbit with the definition of a non-degenerate equilibrium. We can view the equilibrium $x_\ast$ of \eqref{eq: odd number eq} as a periodic solution of arbitrary period $p > 0$. The invariance result of Theorem \ref{lem: preservation} relies on the fact that the time-delay in the controlled system is the same as the period of the periodic orbit. Therefore, to compare the results in this section with the results on periodic orbits, we view the equilibrium $x_\ast$ as a periodic orbit of period $p = T$. In this case, the monodromy operator is given by 
	\begin{equation}
		Y_0(T) = e^{f'(x_\ast) T}.
	\end{equation}
	Therefore, $x_\ast$ is non-degenerate as a periodic orbit of period $T$ if and only if $\frac{2 \pi n i}{T}$ is not an eigenvalue of $f'(x_\ast)$ for every $n \in \mathbb{Z}$. In contrast, for $x_\ast$ to be non-degenerate as an equilibrium, we only require that $0$ is not an eigenvalue of $f'(x_\ast)$. So the assumption that $x_\ast$ is non-degenerate as an equilibrium is milder than the assumption that $x_\ast$ is non-degenerate as a periodic orbit of period $T > 0$.

	\subsection{Corollary: Any-number-resonance limitation for commuting gain matrices with real spectrum}
	
	The next proposition provides an `Any-number limitation' for equilibria, analogous to the statement in Proposition \ref{prop: even number technical} for time-periodic systems. Here we highlight the case in which the linearization has eigenvalues in the right half plane which are in `resonance' with the delay. 
	In the case of commuting gain matrices with real spectrum, we find that the eigenvalues stay on invariant lines parallel to the real axis.
	
	\begin{cor}[Any-number resonance  limitation for equilibria and commuting gain matrices with real spectrum] \label{prop: any number eq}
		Consider the system
		\begin{equation} \label{eq: resonance diagonal ode}
			\dot{x}(t) = f(x(t)), \qquad t \geq 0
		\end{equation}
		with $f: \mathbb{R}^N \to \mathbb{R}^N$ a $C^1$-function. Let $x_\ast \in \mathbb{R}^N$ be an equilibrium of \eqref{eq: resonance diagonal ode} and fix a time delay $T > 0$. Suppose that the time delay is in resonance with one of the unstable eigenvalues, that is,  there exist a real $\lambda_\ast > 0$ and $n \in \mathbb{Z}$ such that
		\begin{equation} \label{eq: resonant ev any number}
			\lambda_\ast + \frac{2 \pi n i}{T} \in \sigma(f'(x_\ast)). 
		\end{equation}
		Assume that $K \in \mathbb{R}^{N \times N}$ satisfies $\sigma(K) \subseteq \mathbb{R}$ and that the gain matrix commutes with the Jacobian, i.e.,
		\begin{equation}
			f'(x_\ast) K = K f'(x_\ast). \label{eq: commute fixed point}
		\end{equation}
		Then $x_\ast$ is unstable as a solution of the controlled system
		\begin{equation} \label{eq: resonance diagonal dde}
			\dot{x}(t) =f(x(t)) + K \left[x(t) - x(t-T) \right]. 
		\end{equation}
	\end{cor}
	\begin{proof}
		We apply Proposition \ref{prop: even number technical} with 
		\begin{equation}
			A(t) := f'(x_\ast). 
		\end{equation} 
		The fundamental solution of the linear equation
		\begin{equation}
			\dot{y}(t) = f'(x_\ast) y(t) 
		\end{equation} 
		is given by 
		\begin{equation}
			Y_0(t) = e^{f'(x_\ast) t}
		\end{equation}  
		and hence $B, P(t)$ in \eqref{eq: B defn}--\eqref{eq: P defn} are given by
		\begin{equation} B: = f'(x_\ast), \qquad P(t) \equiv I. 
		\end{equation}
		Therefore, $P(t) K = K P(t)$ is trivially satisfied and \eqref{eq: commute fixed point} implies that $KB = BK$. Moreover, \eqref{eq: resonant ev any number} implies that the monodromy operator $Y_0(T) = e^{f'(x_\ast) T}$ has an eigenvalue $e^{\lambda_\ast T} > 1$. Thus, Proposition \ref{prop: even number technical} implies that 
		\begin{equation} \dot{y}(t) = f'(x_\ast) y(t) + K \left[y(t) - y(t-T) \right]
		\end{equation}
		has at least one real Floquet multiplier larger than $1$ (or, equivalently, at least one eigenvalue in the strict right half plane). Therefore $x_\ast$ is unstable as a solution of \eqref{eq: resonance diagonal dde}. 
	\end{proof}
	
	We briefly comment on the connection between Corollary \ref{prop: any number eq} and Proposition \ref{prop: even number technical}. In Corollary \ref{prop: any number eq}, we assume that the Jacobian $f'(x_\ast)$ has an eigenvalue $\lambda_\ast + \frac{2 \pi n i}{T}$ with $\lambda_\ast > 0$. Therefore, if we view $x_\ast$ as a periodic orbit of period $T > 0$, the monodromy operator
	\begin{equation}
		Y_0(T) = e^{f'(x_\ast) T}
	\end{equation}
	has a Floquet multiplier $e^{\lambda_\ast T} > 1$. By Proposition \ref{prop: even number technical}, this Floquet multiplier stays on the real line in the controlled system, provided that assumption \eqref{eq: commute fixed point} holds. Therefore, the corresponding \emph{eigenvalue} should stay on one of the lines
	\begin{equation}\label{lines}
		\ell_k := \Big \{ z \in \mathbb{C} \mid \im z = \frac{2 \pi k i}{T} \Big \}.
	\end{equation}
	with $k \in \mathbb{Z}$. 
	However, by continuity of the eigenvalues, the eigenvalue cannot `jump' between lines and thus the eigenvalue should stay on the line $\ell_k$ with $k = n$. So the eigenvalue of the controlled system stays on the same line as the eigenvalue of the uncontrolled system.

	%%%%%%%%%%%%%%%%%%%%%%%%%%%%%%%%%%%%%%%%%%%%%%%%%%%%%%%%%%%%%%%%%%%%%%%%%%%%
	
	\subsection{Corollary: Any-number-resonance limitation on commuting gain matrices with complex spectrum} \label{sec:complex equilibria}
	
	For the case of commuting gain matrices with \emph{complex} spectrum, we prove an `Any-number limitation' analogous to the statement in Proposition \ref{prop:complex gain periodic} for periodic orbits. The limitation also applies when the uncontrolled system has unstable eigenvalues whose imaginary part is in resonance with the time-delay.
	See also the result by Hövel and Schöll \cite{HOE05}, where the case of an unstable focus was addressed in detail.
	Note that in contrast to real control gains, but in agreement with Proposition \ref{prop:complex gain periodic} for periodic orbits, the lines \eqref{lines} are not invariant, see Figure \ref{EigenvalueCurve2}.

	\begin{figure}
		\includegraphics[width=\columnwidth]{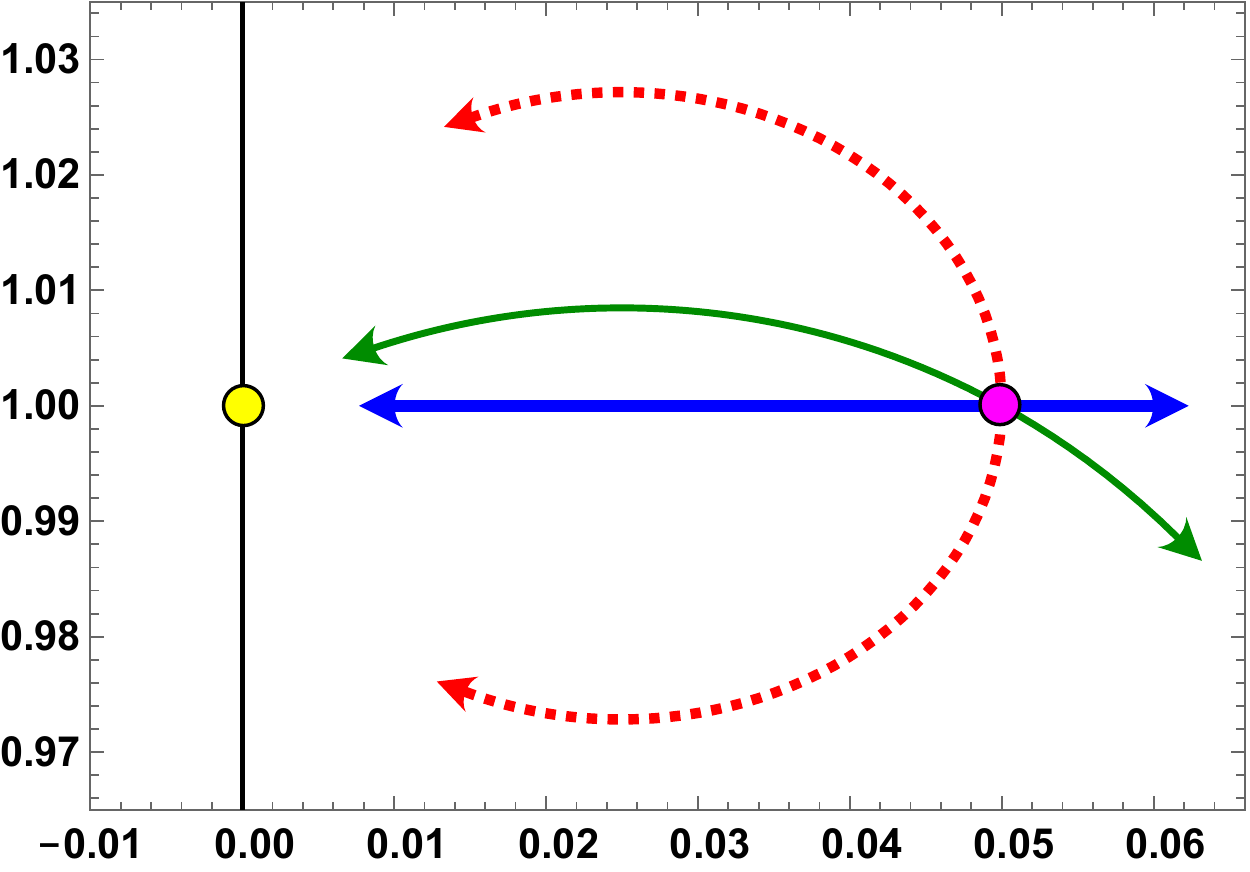}% Here is how to import EPS art
		\caption{\label{EigenvalueCurve2} Imaginary axis and eigenvalues in the complex plane depending of $k_\ast$. Yellow dot: Geometrically invariant resonating center at $i$. Parameter values: $\lambda_\ast=0.05$ (magenta dot), $T= 2 \pi$ $n=1$. Blue: real $k_\ast=\kappa$ (note the invariant line!), red (dotted): purely imaginary $k_\ast$, green (thin): arg$(k_\ast)=e^{0.8\pi i}$. Note that the picture is seemingly asymmetric, because we have not drawn the complex conjugated resonating center.} 
	\end{figure}
	\begin{cor} [Any-number-resonance  limitation for equilibria and commuting gain matrices with complex spectrum]\label{thm: resonance} Consider the system 
		\begin{align} \label{eq: ode 2}
			\dot{x}(t) = f(x(t)), \qquad t \geq 0
		\end{align}
		with $f: \mathbb{R}^N \to \mathbb{R}^N$ a $C^1$-function. Let $x_\ast \in \mathbb{R}^N$ be an equilibrium of \eqref{eq: ode 2} and fix $T > 0$. Suppose there exists a $\lambda_\ast > 0$ and $n \in \mathbb{Z}$ such that 
		\begin{equation} \label{eq: ev complex resonance}
			\lambda_\ast \pm\frac{2 \pi n i}{T} \in \sigma(f'(x_\ast)).
		\end{equation}
		Assume that $K \in \mathbb{R}^{N \times N}$ commutes with the Jacobian, i.e. 
		\begin{equation} \label{eq: commute fixed point 2}
			f'(x_\ast) K = K f'(x_\ast). 
		\end{equation}
		Then $x_\ast$ is unstable as a solution of the controlled system
		\begin{equation} \label{eq: pyragas2}
			\dot{x}(t) = f(x(t)) + K \left[x(t) - x(t-T) \right].
		\end{equation}
	\end{cor}
	\begin{proof}
		We apply Proposition \ref{prop:complex gain periodic} with 
		\begin{equation}
			A(t) := f'(x_\ast). 
		\end{equation} 
		The fundamental solution of the linear equation
		\begin{equation}
			\dot{y}(t) = f'(x_\ast) y(t) 
		\end{equation} 
		is given by 
		\begin{equation}
			Y_0(t) = e^{f'(x_\ast) t}
		\end{equation}  
		and hence $B, P(t)$ in \eqref{eq: B defn}--\eqref{eq: P defn} are given by
		\begin{equation} B: = f'(x_\ast), \qquad P(t) \equiv I. 
		\end{equation}
		Therefore, $P(t) K = K P(t)$ is trivially satisfied and \eqref{eq: commute fixed point 2} implies that $KB = BK$. Moreover, \eqref{eq: ev complex resonance} implies that the monodromy operator $X_0(T) = e^{f'(x_\ast) T}$ has an eigenvalue $e^{\lambda_\ast T} > 1$. Thus, Proposition \ref{prop: even number technical} implies that 
		\begin{equation} \dot{y}(t) = f'(x_\ast) y(t) + K \left[y(t) - y(t-T) \right]
		\end{equation}
		has at least one Floquet multiplier outside the unit circle (or, equivalently, at least one eigenvalue in the strict right half plane). Therefore $x_\ast$ is unstable as a solution of \eqref{eq: pyragas2}. 
	\end{proof}
	
	As an application, we consider 2-dimensional systems where both the Jacobian and the control matrix have two complex conjugated eigenvalues. In this case, stabilization is impossible if the time delay is chosen in resonance with the eigenvalues of the Jacobian, see also \cite{HOE05}. 
	\begin{cor}
		Consider the system 
		\begin{align} \label{eq: ode 3}
			\dot{x}(t) = f(x(t)), \qquad t \geq 0
		\end{align}
		with $f: \mathbb{R}^2 \to \mathbb{R}^2$ a $C^1$-function. Let $x_\ast \in \mathbb{R}^N$ be an equilibrium of \eqref{eq: ode 3}. Suppose that the Jacobian $f'(x_\ast)$ is given by
		\begin{equation}
			f'(x_\ast) = \begin{pmatrix}
				\lambda & - \omega \\
				\omega & \lambda
			\end{pmatrix}
		\end{equation}
		with $\lambda > 0$ and $\omega > 0$. Moreover, let $K$ be given by
		\begin{equation} \label{eq:K ev}
			K = \begin{pmatrix} \alpha & - \beta \\
				\beta & \alpha
			\end{pmatrix}
		\end{equation}
		with $\alpha, \beta \in \mathbb{R}$. Then, for $n \in \mathbb{N}$, $x_\ast$ is unstable as a solution of the controlled system
		\begin{equation}
			\dot{x}(t) = f(x(t)) + K \left[x(t) - x\left(t- \frac{2 \pi n}{\omega} \right) \right]. 
		\end{equation}
	\end{cor}
	\begin{proof}
		The spectrum of $f'(x_\ast)$ is given by 
		\begin{equation}
			\sigma(f'(x_\ast)) = \{ \lambda \pm i \omega \}. 
		\end{equation}
		Moreover, with $K$ as in \eqref{eq:K ev}, we have that $f'(x_\ast) K  =K f'(x_\ast)$. Thus, if we apply Corollary \ref{thm: resonance} with $T = \frac{2 \pi n}{\omega}$ the claim follows. 
	\end{proof}
	
	Another application is given by a ring of $n$ coupled Stuart-Landau oscillators, which shows a multitude of ponies-on-a-merry-go-round-solutions. It has been proven\cite{SCH16} that only the fully synchronized orbit can be stabilized via Pyragas control. This is a direct consequence of Corollary \ref{thm: resonance}: The time delay $2\pi$ is prescribed by the periodic orbit and in resonance with the imaginary part $1$ of all complex eigenvalues near Hopf bifurcation. In the linearization, the system decouples into $n$ complex equations.  For all periodic orbits except for the synchronized one, at least one of these equations will have an empty resonating center at $i$ together with an eigenvalue with positive real part \cite{SCH16}. Therefore, Corollary \ref{thm: resonance} directly forbids stabilization via Pyragas control, at equivariant Hopf bifurcation. The result in\cite{SCH16} shows in addition that this failure of Pyragas control persists close to the Hopf bifurcation point. 
	%%%%%%%%%%%%%%%%%%%%%%%%%%%%%%%%%%%%%%%%%%%%%%%%%%%%%%%%%%%%%%%%
	\section{Conclusion}
	
	We proved two fundamental invariance principles of Pyragas control for periodic orbits and equilibria.  A number of limitations on Pyragas stabilization of periodic solutions and equilibria follow. 
	Compared to previous literature on this subject\cite{NAK97,JUS99,HOO12}, our approach provides a new perspective on two crucial points. First of all, in our analysis, we emphasized the geometric rather than the algebraic aspects of the Floquet or eigenvalue problem. 
	Instead of purely calculating algebraic multiplicities, we count the geometric multiplicity, that is, the dimension of the eigenspace. 
	The geometric multiplicity of the determining and resonating centers stay invariant under Pyragas control, while the algebraic number of eigenvalues does not\cite{FIE07,JUS07}. Moreover, in our formulation of limitations to control, we emphasised properties of the unstable object itself, rather than the properties of the uncontrolled dynamical systems. For controllability, it matters whether there exists an determining or resonating center, but its provenance is irrelevant. These two shifts of perspective provide a clear and unifying understanding of the previously often misinterpreted odd-number limitation, and moreover lead naturally to `any-number limitations' for commutative gain matrices.

	\begin{acknowledgments}
		This work was partially supported by SFB 910 “Control of self-determining nonlinear systems: Theoretical methods and concepts of application”, project A4: “Spatio-temporal patterns: observation, control, and design”. The work of BdW was supported by the Berlin Mathematical School (BMS).  We are grateful to Prof. Dr. Bernold Fiedler and Prof. Dr. Sjoerd Verduyn Lunel for their constant support and encouragement. We thank Dr. Jia-Yuan Dai and Alejandro L\'opez Nieto  for many fruitful discussions and helpful remarks.

	\end{acknowledgments}

	\appendix
	\section{Floquet theory}
	
	In Section \ref{sec:real periodic}, we have  crucially used Floquet theory to transform the linear, time-periodic ODE $\dot{y}(t) = A(t) y(t)$ into a linear, autonomous ODE. 
	For the comfort of the reader, we include a short summary for Floquet theory for ordinary and delay differential equations. A more general treatment for Floquet theory for ODE can, for example, be found in \cite[Chapter III.7]{haleODE}; Floquet theory for DDE is extensively treated in \cite[Chapter XIII]{DIE95} and \cite[Chapter 8]{HAL93}. 
	
	\subsection{Floquet theory for ODE} \label{sec:appendix ode}
	
	Consider the linear, time-periodic system
	\begin{equation} \label{eq: A appendix}
		\dot{y}(t) = A(t) y(t), \qquad t \geq 0
	\end{equation}
	with $y(t) \in \mathbb{R}^N$ and $A(t) \in \mathbb{R}^{N \times N}$. Assume that there exists a $T > 0$ such that $A(t+T) = A(t)$, i.e., the system \eqref{eq: A appendix} is $T$-periodic. Denote by $Y_0(t) \in \mathbb{R}^{N \times N}, \ t \geq 0$ the fundamental solution of \eqref{eq: A appendix}, i.e., $Y_0(t)$ satisfies
	\begin{align} \label{eq: fund sol appendix}
		\begin{cases}
			\frac{d}{dt} Y_0(t) = A(t) Y_0(t), \qquad t > 0 \\
			Y_0(0) = I. 
		\end{cases}
	\end{align}
	We refer to the operator $Y_0(T): \mathbb{R}^N \to \mathbb{R}^N$ as the \emph{monodromy operator} of \eqref{eq: A appendix}; we refer to the eigenvalues of $Y_0(T)$ as the \emph{Floquet multipliers} of system \eqref{eq: A appendix}. 
	
	Since $Y_0(t)$ is invertible for all $t \geq 0$, we in particular have that $0 \not \in \sigma(Y_0(T))$. Therefore, there exists a (non-unique!) matrix $B \in \mathbb{C}^{N \times N}$ such that 
	\begin{equation}
		Y_0(T) = e^{BT}. \label{eq: B appendix}
	\end{equation}
	The following theorem says that there exists a time-periodic transformation that transforms system \eqref{eq: A appendix} into an autonomous, linear system. 
	
	\begin{theorem}[Floquet] Consider system \eqref{eq: A appendix} with $A(t+T) = A(t)$ for all $t \in \mathbb{R}$. Let $Y_0(t)$ be as in \eqref{eq: fund sol appendix} and let $B \in \mathbb{C}^{N \times N}$ be as in \eqref{eq: B appendix}. 
		
		Then the map
		\begin{equation} \label{eq: P appendix}
			P(t) : = Y_0(t) e^{-Bt}
		\end{equation}
		satisfies $P(t+T) = P(t)$ for all $t \in \mathbb{R}$. Moreover, the coordinate transformation $y(t) = P(t)v(t)$ transforms system \eqref{eq: A appendix} into the system
		\begin{equation} \label{eq: B ode appendix}
			\dot{v}(t) = B v(t).
		\end{equation}
	\end{theorem}
	\begin{proof}
		\emph{Step 1:} We first prove that $P(t)$ as defined in \eqref{eq: P appendix} is $T$-periodic. Indeed, we have that 
		\begin{align}
			P(t+T) &= Y_0(t+T) e^{-BT} e{- Bt} \\
			&= Y_0(t+T) Y_0(T)^{-1} e^{-Bt} \\
			&= Y_0(t) e^{-Bt} = P(t). 
		\end{align}
		\emph{Step 2:} We now prove that the transformation $y(t) = P(t) v(t)$ transforms system \eqref{eq: A appendix} into system \eqref{eq: B ode appendix}. Differentiating \eqref{eq: P appendix} gives that 
		\begin{align}
			\frac{d}{dt} P(t) &= A(t) Y_0(t) e^{-Bt} - Y_0(t) e^{-Bt} B \\
			&=A(t) P(t) - P(t) B.
		\end{align}
		If $y(t)$ is a solution of \eqref{eq: A appendix}, then $y(t) = P(t) v(t)$ satisfies
		\begin{align}
			A(t) P(t) v(t) &= \left( \frac{d}{dt} P(t) \right) v(t) + P(t) \frac{d}{dt} v(t) \\
			&= A(t) P(t) v(t) - P(t) B v(t) + P(t) \frac{d}{dt} v(t).
		\end{align}
		which implies 
		\begin{equation}
			P(t) \frac{d}{dt} v(t) = P(t) B v(t). 
		\end{equation}
		Since $P(t)$ is invertible for all $t \in \mathbb{R}$, this implies that $v$ satisfies \eqref{eq: B ode appendix}. 
	\end{proof}
	
	Let $y_0$ be an eigenvector of $Y_0(T) = e^{BT}$ with eigenvalue $\mu$. Then the decomposition 
	\begin{equation} \label{eq: per exp ode}
		Y_0(t) = P(t) e^{Bt}
	\end{equation}
	implies that $y(t) : = P(t) e^{Bt} y_0$ is a solution of \eqref{eq: A appendix} that satisfies $y(t+T) = \mu y(t)$. Vice versa, if $y(t)$ is a solution of  \eqref{eq: A appendix}, then $y(t) = Y_0(t) y(0)$; so if $y(t+T) = \mu y(t)$, then $y(0)$ is an eigenvalue of $Y_0(T) = e^{BT}$ with eigenvalue $\mu$. Hence we conclude:
	
	\begin{lemma} \label{lem: count ev ode}
		Let $\mu \in \sigma(X(T))$. Then there is a one-to-one correspondence between eigenvectors $y_0$ of $Y_0(T)$ with eigenvalue $\mu$ and solutions of \eqref{eq: A appendix} with $y(t+T) = \mu y(t)$. 
	\end{lemma}
	
	\subsection{Floquet theory for DDE} \label{sec:appendix dde}
	We recall the main results of Floquet theory for DDE, focussing on the DDE
	\begin{equation} \label{eq: dde appendix}
		\dot{y}(t) = A(t) y(t) + K \left[y(t) - y(t-T)\right]
	\end{equation}
	with $T > 0$, $K \in \mathbb{R}^{N \times N}$ and $A(t+T) = A(t)$ for all $t \in \mathbb{R}$. 
	
	For a general linear time-periodic, linear DDE, the non-zero spectrum of the monodromy operator consists of eigenvalues of finite multiplicity, see \cite[Chapter VIII]{DIE95}. The next lemma proves this for the DDE \eqref{eq: dde appendix}. 
	
	\begin{lemma} 
		Let $Y_1(t), t \geq 0$ be the fundamental solution of system \eqref{eq: dde appendix}. Then the operator
		\begin{equation}
			Y_1(T): C \left([-T, 0], \mathbb{R}^n \right) \to C \left([-T, 0], \mathbb{R}^n \right)
		\end{equation}
		is compact and hence the non-zero spectrum consists of isolated eigenvalues of finite algebraic multiplicity. 
	\end{lemma}
	\begin{proof}
		Let $Y_A(t), t \geq 0$ be the fundamental solution of the ODE
		\begin{equation}
			\dot{y}(t) = A(t) y(t) + K y(t)
		\end{equation}  
		with $Y_A(0) = I$. 
		Fix $\phi \in C \left([-T, 0], \mathbb{R}^N \right)$ and consider the initial value problem
		\begin{align}
			\begin{cases}
				\dot{y}(t) = A(t) y(t) + K \left[y(t) - y(t-T) \right], \qquad &t \in [0, T] \\
				y(t) = \phi(t), \qquad &t \in [-T, 0]
			\end{cases}
		\end{align}
		Since $y(t-T) = \phi(t-T)$ for $t \in [0, T]$, Variation of Constants gives that 
		\begin{equation}
			y(t) = Y_A(t) \phi(0) - \int_0^t Y_A(t) Y_A(s)^{-1} K \phi(s-T) ds
		\end{equation}  
		or, for $\theta \in [-T, 0]$, 
		\begin{widetext}
			\begin{align}
				y(T+\theta) & = Y_A(T+\theta) \phi(0) - \int_{0}^{T+\theta} Y_A(T+\theta) Y_A(s)^{-1} K \phi(s-T) ds \\
				&= Y_A(T+\theta) \phi(0) - \int_{-T}^s Y_A(T+\theta) Y_A(T+s)^{-1} K \phi(s) ds \\
				& = Y_A(T+\theta) \phi(0) - \int_{-T}^s Y_A(\theta) Y_A(s)^{-1} K \phi(s) ds.
			\end{align}
		\end{widetext}
		Hence the monodromy operator is given by 
		\begin{equation}
			(Y_1(T) \phi)(\theta) = Y_A(T+\theta) \phi(0) - \int_{-T}^s Y_A(\theta) Y_A(s)^{-1} K \phi(s) ds.
		\end{equation} 
		The Arzelà-Ascoli theorem implies that $X(T)$ is a compact operator, and hence its non-zero spectrum consists of isolated eigenvalues of finite multiplicity. 
	\end{proof}
	
	For time-periodic ODE, the fundamental solution can be written as the product of a time-periodic and an exponential factor, cf. \eqref{eq: per exp ode}. For a general DDE, one cannot obtain such a decomposition (see \cite[Chapter VIII, Exercise 4.3]{DIE95}. However, a statement similar to Lemma \ref{lem: count ev ode} still holds:
	
	\begin{lemma}
		Let $Y_1(T)$ be the monodromy operator of \eqref{eq: dde appendix} and let $\mu \in \sigma_{pt}(X(T))$. Then there is a one-to-one correspondence between eigenfunctions $\phi$ of $Y_1(T)$ with eigenvalue $\mu$ and solutions of \eqref{eq: dde appendix}  with $y(t+T) = \mu y(t)$.
	\end{lemma}
	\begin{proof}
		Let $\phi$ be an eigenfunction of $Y_1(T)$ associated to the eigenvalue $\mu$. Then $y_t : = Y_1(t) \phi$ is a solution of \eqref{eq: dde appendix} satisfying $y_T = \mu y_0$ and hence, by uniqueness of solutions $y_{t+T} = \mu y_t$.  
		Vice versa, suppose that $x(t)$ is a solution of \eqref{eq: dde appendix}, then $x_t = Y_1(t) x_0$. So, if $y(t+T) = y(t)$ for all $t \in \mathbb{R}$, then in particular $y_T = \mu y_0$ and hence $Y_1(T) x_0 = \mu y_0$. 
	\end{proof}
	
	\section*{Data Availability Statement}
	Data sharing is not applicable to this article as no new data were created or analyzed in this study.
	
	\section*{References}
	\bibliography{bib_dissertation}% Produces the bibliography via BibTeX.
	
\end{document}